\documentclass[12pt]{amsart}
\usepackage{amssymb}
\usepackage[margin=1in]{geometry}

\title{Kirchberg's factorization property for locally compact groups}
\author{Matthew Wiersma}
\address{Mathematical and Statistical Sciences, University of Alberta, Edmonton, Alberta, T6G 2G1, Canada}
\email{mwiersma@uwalberta.ca}

\newtheorem{theorem}{Theorem}[section]
\newtheorem{corollary}[theorem]{Corollary}
\newtheorem{prop}[theorem]{Proposition}
\newtheorem{lemma}[theorem]{Lemma}

\theoremstyle{remark}
\newtheorem{remark}[theorem]{Remark}
\newtheorem{claim}{Claim}

\newenvironment{claimproof}[1]{\par\noindent\emph{Proof of Claim}.\space#1}{\hfill $\blacksquare$\vspace{0.5em}}

\theoremstyle{definition}

\newtheorem{example}[theorem]{Example}
\newtheorem{prob}[theorem]{Problem}

\newcommand{\fn}{\!:}
\newcommand{\C}{\mathbb C}
\newcommand{\R}{{\mathbb R}}

\newcommand{\Hi}{\mathcal{H}}
\newcommand{\lla}{\left\langle}
\newcommand{\rra}{\right\rangle}
\newcommand{\mc}{\mathcal}

\newcommand{\tn}{\textnormal}

\newcommand{\F}{\mathbb{F}}
\newcommand{\supp}{\mathrm{supp}}

\newcommand{\mr}{\mathrm}

\begin{document}

\begin{abstract}
	A locally compact group $G$ has the {\it factorization property} if the map
	$$\mr C^*(G)\odot \mr C^*(G)\ni a\otimes b\mapsto \lambda(a)\rho(b)\in\mc B(\mr L^2(G))$$
	is continuous with respect to the minimal C*-norm. This paper seeks to initiate a rigorous study of this property in the case of locally compact groups which, in contrast to the discrete case, has been relatively untouched.
	A partial solution to the question of when the factorization property passes to continuous embeddings is given -- a question which traces back to Kirchberg's seminal work on the topic and is known to be false in general. It is also shown that every ``residually amenably embeddable'' group must necessarily have the factorization property and that an analogue of Kirchberg's characterization of the factorization property for discrete groups with property (T) holds for a more general class of groups.
\end{abstract}

\maketitle

\section{Introduction}

The history of Kirchberg's factorization property dates back to 1976 when Wasserman gave the first example of a non-exact C*-algebra by showing that the sequence
$$0\to \mr C^*(\F_2)\otimes_{\min}J\to \mr C^*(\F_2)\otimes_{\min}\mr C^*(\F_2)\to \mr C^*(\F_2)\otimes_{\min} \mr C^*_r(\F_2)\to 0$$
is not exact, where $J$ denotes the kernel of the canonical map $\mr C^*(\F_2)\to \mr C^*_r(\F_2)$ (see \cite{wasserman}). The key step in this proof was showing that the positive definite function $1_\Delta$ (the characteristic function of the diagonal subgroup $\Delta=\Delta_{\F_2}$ of $\F_2\times\F_2$) on $\F_2\times\F_2$ extends to a positive linear functional on $\mr C^*(\F_2)\otimes_{\min}\mr C^*(\F_2)$. Inspired by this technique, Kirchberg defined a locally compact group $G$ to have the {\it factorization property} (or {\it property (F)}) if the representation of $\lambda\cdot\rho\fn G\times G\to \mc B(\mr L^2(G))$ defined by $(\lambda\cdot \rho)(s,t)=\lambda(s)\rho(t)$ extends to a $*$-representation of $\mr C^*(G)\otimes_{\min}\mr C^*(G)$ (see \cite{kirch1}), i.e., if the map $\mr C^*(G)\odot \mr C^*(G)\to \mc B(\mr L^2(G))$ defined by $a\otimes b\mapsto\lambda(a)\rho(b)$ is continuous with respect to the minimal C*-norm. Equivalently, $G$ has the factorization property if and only if $\lambda\cdot \rho$ is weakly contained in $\pi_u\times\pi_u$, where $\pi_u$ denotes the universal representation of $G$. Note that if $G$ is a discrete group, then $G$ has the factorization property if and only if the positive definite function $1_\Delta$ on $G\times G$ extends to a positive linear functional on $\mr C^*(G)\otimes_{\min}\mr C^*(G)$ since the GNS representation of $1_\Delta$ is $\lambda\cdot \rho$. Wasserman's proof then applies to show that a discrete group $G$ with the factorization property is amenable if and only if the sequence
\begin{equation}\label{Eqn:intro}
0\to \mr C^*(G)\otimes_{\min}J\to \mr C^*(G)\otimes_{\min}\mr C^*(G)\to \mr C^*(G)\otimes_{\min} \mr C^*_r(G)\to 0
\end{equation}
is exact, where $J$ is the kernel of the canonical map $\mr C^*(G)\to \mr C^*_r(G)$.

Since Kirchberg's seminal paper on the subject, discrete groups with the factorization property have been studied in connection to a variety of different topics. For instance, finitely generated discrete groups with the factorization property are of interest from the perspective of geometric group theory because they lie strictly between the classes of finitely generated residually amenable discrete groups and finitely generated hyperlinear groups (see \cite{thom}). From the operator algebras perspective, groups with the factorization property are studied in connection to various properties of C*-algebras and their connection to Kirchberg's reformulation of the Connes Embedding Conjecture (see \cite{brown} and \cite{ozawa}).

In contrast to the discrete case, the study of the factorization property for locally compact groups remains relatively untouched. There two probable reasons for this disparity. The first is that the analogues of important results for discrete groups with the factorization property fail in the locally compact case by virtue of the fact that $\mr C^*(G)$ may be nuclear for nonamenable locally compact groups $G$. For example, sequence \ref{Eqn:intro} will be exact whenever $G$ is a locally compact group such that $\mr C^*(G)$ is nuclear. The second reason is that proofs of results about the factorization property tend to rely on the fact that a discrete group $G$ has the factorization property if and only if $1_\Delta$ extends to a positive linear functional on $\mr C^*(G)\otimes_{\min}\mr C^*(G)$. This paper seeks to initiate a rigorous study of the factorization property in the context of locally compact groups. The first of the two issues mentioned above is addressed by considering particular classes of locally compact groups (namely the classes of QSIN and inner amenable groups) which are large enough to be interesting, but well enough behaved that proper analogues of results from the discrete case hold. We also show that it is often possible to get around the second issue, but proofs tend to become substantially more difficult.

This paper is structured as follows. After a brief section addressing notation and background material, Section \ref{Sec:basic} relates the factorization property with properties of group C*-algebras, and considers some stronger properties than the factorization property. Section \ref{Sec:subgroup} provides a partial answer to when the factorization property passes to continuous embeddings -- a question which traces back to Kirchberg's seminal work on the topic and is known to be false in general. Section \ref{Sec:residual} studies residual properties of groups with the factorization property and, in particular, shows that ``residually amenably embeddable'' groups have the factorization property. Next, Section \ref{Sec:T} generalizes Kirchberg's characterization of the factorization property for discrete groups with property (T) (see \cite{kirch2}) to a more general class of groups. The paper is then concluded by posing two problems for future research.

\section{Notation and Background}

\subsection{Notation and conventions}

Given a Banach space $X$, we will let $\mc B(X)$ denote the space of bounded linear maps from $X$ to itself, and $X_1$ denote the set of elements in $X$ which have norm 1. In the case that we are dealing with a C*-algebra $A$, we will let $A_+$ denote the positive cone in $A$. All groups $G$ will be assumed to be locally compact unless otherwise stated. The left and right regular representations of a group $G$ are denoted by $\lambda$ and $\rho$, respectively. For $p\in [1,\infty]$, we let $\tau_p\fn G\to \mr B(\mr L^p(G))$ denote the isometric conjugation action given by $\tau_p(s)f(t)=f(s^{-1}ts)\Delta(s)^{1/p}$, where $\Delta$ denotes the modular function of $G$. All group representations in this paper will be assumed to be unitary and continuous in the strong operator topology. If $\pi\fn G\to \mc B(\Hi)$ is a representation of a locally compact group $G$ and $\xi,\eta\in\Hi$, we will let $\pi_{\xi,\eta}\fn G\to \C$ denote the coefficient function defined by $\pi_{\xi,\eta}(s)=\lla \pi(s)\xi,\eta\rra$.

Suppose that $G$ is a locally compact group. The left Haar measure of $G$ will be denoted by $m=m_G$ or simply by $dx$. If $N$ is a closed normal subgroup of $G$, then $\dot x$ will denote the coset $xN\in G/N$. We will always make the assumption that the left haar measures on $G$, $G/N$ and $N$ are normalized so that
$$ \int_G f(x)\,dx=\int_{G/N}\int_N f(xn)\,dn\,d\dot x $$
for every $f\in \mr C_c(G)$.

\subsection{Some classes of locally compact groups}

In this section we briefly introduce the notions of SIN groups, QSIN groups, and inner amenable groups since many operator algebraists may be unfamiliar with these classes of locally compact groups.

A locally compact group is a {\it small invariant neighbourhood group} or {\it SIN group} if the identity $e$ of $G$ admits a neighbourhood base of conjugation invariant compact sets. Examples of SIN groups include all abelian, compact, and discrete groups. A well known characterization due to Mosak states that a locally compact group $G$ is SIN if and only if $\mr L^1(G)$ admits a central bounded approximate identity, i.e., a bounded approximation identity $\{e_\alpha\}$ so that $\tau_1(s)e_\alpha=e_\alpha$ for every $s\in G$ and index $\alpha$ (see \cite{mosak}).

Offering a generalization by way of Mosak's characterization of SIN groups, a {\it quasi-SIN group} or {\it QSIN group} is a locally compact group $G$ for which $\mr L^1(G)$ admits a quasi-invariant bounded approximate identity, i.e., a bounded approximate identity $\{e_\alpha\}$ such that $\|\tau(s)e_\alpha-e_\alpha\|\to 0$ uniformly on compact subsets of $G$. The class of QSIN groups is much more general than that of SIN groups and, in particular, contains every amenable group by a result of Losert and Rindler (see \cite{lr}). Interested readers are encouraged to see \cite{stokke} for a nice treatment of the basic theory of QSIN groups.

Finally, a locally compact group $G$ is {\it inner amenable} if $\mr L^\infty(G)$ admits a conjugation invariant mean, i.e., a state $\mu\in \mr L^\infty(G)^*$ so that
$$\mu(s\cdot f)=\mu(f)$$
for every $s\in G$ and $f\in \mr L^\infty(G)$, where the action of $s$ on $\mr L^\infty(G)$ is given by $(s\cdot f)(t)=f(s^{-1}ts)$. Inner amenable groups are the most general of the three classes of locally compact groups introduced in this section and, in particular, every discrete group is inner amenable. There is another notion of inner amenability for discrete groups where not every discrete group is inner amenable, but we will not work with this notion.

\section{Basic results}\label{Sec:basic}

Kirchberg initiated the study of the factorization property due to the properties of the associated group C*-algebras in the discrete case. Though the analogues of these results do not hold for all locally compact groups, we show that they do for the class of inner amenable groups.

Recall that a representation $\pi\fn G\to \mc B(\Hi)$ of a locally compact group $G$ is {\it amenable} if there exists a state $\mu\in \mc B(\Hi)^*$ such that
$$ \mu(\pi(s)T\pi(s^{-1}))=\mu(T) $$
for every $s\in G$ and $T\in \mc B(\Hi)$. Key facts which will be used throughout this section are that a locally compact group $G$ is
\begin{itemize}
	\item amenable if and only if $\lambda$ is amenable,
	\item inner amenable if and only if $\tau_2$ is amenable.
\end{itemize}
All results on amenable representations used in this paper can be found in Bekka's original paper \cite{bekka}.

\begin{prop}\label{Prop:seq}
	Suppose that $G$ is an inner amenable group which admits the factorization property and let $J$ denote the kernel of the canonical map from $\mr C^*(G)$ onto $\mr C^*_r(G)$. Then the sequence
	\begin{equation}\label{seq}
		0\to \mr C^*(G)\otimes_{\min} J\to \mr C^*(G)\otimes_{\min} \mr C^*(G)\to \mr C^*(G)\otimes_{\min} \mr C^*_r(G)\to 0
	\end{equation}
	is exact if and only if $G$ is amenable.
\end{prop}

\begin{proof}
It is clear that Sequence \ref{seq} is exact whenever $G$ is amenable by virtue of $\mr C^*(G)$ being nuclear. So we will assume that the above sequence is exact and deduce that $G$ is amenable.
Since $G$ has the factorization property, the representation $\lambda\cdot \rho$ extends to a $*$-representation $\mr C^*(G)\otimes_{\min} \mr C^*(G)$. Viewed as such, the kernel of $\lambda\cdot \rho$ contains $\mr C^*(G)\otimes_{\min} J$ and, thus, $\lambda\cdot\rho$ extends to a $*$-representation of $\mr C^*(G)\otimes_{\min} \mr C^*_r(G)$ since Sequence \ref{seq} is exact and, so,  $\lambda\cdot\rho$ is weakly contained in $\lambda\times\lambda$. In particular, this implies that $(\lambda\cdot \rho)|_\Delta$ is weakly contained in $(\lambda\times\lambda)|_\Delta$, where $\Delta$ denotes the diagonal subgroup of $G\times G$. Observe that $(\lambda\cdot\rho)|_\Delta=\tau_2$ and $(\lambda\times \lambda)|_\Delta=\lambda\otimes\lambda$ is unitarily equivalent to an amplification of $\lambda$ by Fell's absorption principle. So $\tau_2$ is weakly contained $\lambda$. Then $\lambda$ is an amenable representation since $\tau_2$ is an amenable representation and, hence, $G$ is an amenable group.
\end{proof}


This result shows for such groups $G$ that $\mr C^*(G)$ is exact if and only if $G$ is amenable. This result can be improved by appealing to the following two results of Effros and Haagerup.
The definitions of local reflexivity and the local lifting property (LLP) (which may be found in \cite{pisier}) are omitted from this paper due to their technical nature and since they do not play a role in the remainder of this paper. We do, however, mention that every locally reflexive C*-algebra is exact and the LLP for a C*-algebra $A$ is equivalent to the condition that $A\otimes_{\min}\mc B(\Hi)=A\otimes_{\max}\mc B(\Hi)$ canonically.

\begin{theorem}[Effros-Haagerup {\cite[Proposition 5.3]{effh}}]
	The sequence
	$$0 \to J \otimes_{\min} C \to A \otimes_{\min} C \to A/J \otimes_{\min}  C \to 0$$
	is exact for every locally reflexive C*-algebra $A$, closed two-sided ideal $J$ of $A$, and every C*-algebra $C$.
\end{theorem}
\begin{theorem}[Effros-Haagerup {\cite[Theorem 3.2]{effh}}]
	Let $B$ be a C*-algebra and $J$ a closed two sided ideal of $B$. If $A:=B/J$ has the LLP, then the sequence
	$$0 \to  J \otimes _{\min} C \to B \otimes _{\min} C \to B/J \otimes _{\min} C \to 0$$
	is exact for every C*-algebra $C$.
\end{theorem}

\begin{corollary}
	Let $G$ be an inner amenable group with the factorization property. The following are equivalent.
	\begin{enumerate}
		\item[(i)] $G$ is amenable,
		\item[(ii)] $\mr C^*(G)$ is locally reflexive,
		\item[(iii)] $\mr C^*_r(G)$ has the LLP.
	\end{enumerate}
\end{corollary}

\subsection*{Related properties}

We finish off this section by briefly considering two properties related to the factorization property. Namely, when the representation $\lambda\cdot\rho$ extends to a $*$-representation of
\begin{itemize}
	\item[(a)] $\mr C^*_r(G)\otimes_{\min}\mr C^*_r(G)$, or
	\item[(b)] $\mr{VN}(G)\otimes_{\min}\mr{VN}(G)$.
\end{itemize}
For convenience, we will say that groups $G$ possessing the first of these two properties have {\it Property (F2)} and those possessing the latter of the two properties have {\it Property (F3)}. Then
$$ \mr{Property\,(F3)}\Rightarrow\mr{Property\,(F2)}\Rightarrow\mr{Property\,(F)}.$$
Indeed, the first implication follows from the fact that $\mr C^*_r(G)\otimes_{\min}\mr C^*_r(G)$ is a C*-subalgebra of $\mr{VN}(G)\otimes_{\min}\mr{VN}(G)$, and the second implication because $\mr C^*_r(G)\otimes_{\min}\mr C^*_r(G)$ is a canonical quotient of $\mr C^*(G)\otimes_{\min}\mr C^*(G)$.

The following proposition relates properties (F2) and (F3) to other well known properties.

\begin{prop}
	Let $G$ be a locally compact group.
	\begin{enumerate}
		\item[(a)] $G$ has property (F3) if and only if $\mr{VN}(G)$ is injective. In particular, $G$ has property (F3) if $G$ is amenable.
		\item[(b)] If $G$ is inner amenable, then $G$ has property (F3) if and only if $G$ is amenable.
		\item[(c)] If $G$ is inner amenable, then $G$ has property (F2) if and only if $G$ is amenable.
	\end{enumerate}
\end{prop}

\begin{proof}
	(a) Let $\mr{VN}_\rho(G)=\rho(G)''\subset \mc B(\mr L^2(G))$ denote the von Neumann algebra associated to the right regular representation of $G$. Then $\mr{VN}_\rho(G)=\mr{VN}(G)'$. So $\mr{VN}(G)$ is injective if and only if the multiplication map $\mr{VN}(G)\odot \mr{VN}_\rho(G)\to \mc B(\mr L^2(G))$ is continuous with respect to the minimal C*-norm (see \cite[Lemma 2.1]{effl}). Since $\mr{VN}_\rho(G)\cong \mr{VN}(G)$ canonically, we conclude that $\mr{VN}(G)$ is injective if and only if $G$ has property (F3).
	
	(b) This follows immediately from part (a) and a result of Lau and Paterson which states that a locally compact group $G$ is amenable if and only if $G$ is inner amenable and $\mr{VN}(G)$ is injective (see \cite[Corollary 3.2]{laup}).
	
	(c) Suppose $G$ is an inner amenable group with property (F2). Then $\lambda\cdot \rho$ is weakly contained in the $\lambda\times\lambda$. Restricting to the diagonal subgroup of $G\times G$, we get that $\tau_2$ is weakly contained in $\lambda$ by a similar argument as used in the proof of Proposition \ref{Prop:seq}. Since $\tau_2$ is an amenable representation, we deduce that $\lambda$ is an amenable representation and, hence, that $G$ is amenable.
\end{proof}

\begin{corollary}
	Suppose that $G$ is a locally compact group such that $\mr{VN}(G)$ is injective. Then $G$ has the factorization property.
\end{corollary}

The above proposition shows that every separable almost connected group has property (F3) by deep work of Connes (see \cite[Corollary 6.9]{connes}), and that properties (F2) and properties (F3) differ from property (F) since every residually finite discrete group has the factorization property. It is not known to the author whether properties (F2) and (F3) coincide.

%
%

%

\section{Hereditary properties}\label{Sec:subgroup}

Let $G$ and $H$ be locally compact groups. We say that $H$ {\it continuously embeds} into $G$ if there exists a continuous injective group homomorphism $\iota\fn H\to G$.
The main purpose of this section is to address the problem of when the factorization property passes to continuous embeddings. The first and motivating attempt at this problem was made by Kirchberg in \cite[Corollary 7.3 (iii)]{kirch1} which states that if $\mr C^*(G)$ is nuclear and $H$ is discrete, then $H$ also has the factorization property. Unfortunately, the proof of this result contains an error and a counterexample to the result was produced by Thom in \cite{thom}. Thom was, however, able to recover Kirchberg's result when $G$ is assumed to be a unimodular almost connected group which admits a neighbourhood base $\{E_\alpha\}_\alpha$ for the identity such that
$$ \frac{m_G(sE_\alpha\Delta E_\alpha s)}{m_G(E_\alpha)}\to 0$$
for every $s\in G$ (see \cite[Remark 3.1]{thom}). The class of locally compact groups $G$ which admit such a neighbourhood base are known in the literature as almost-SIN groups and coincides with the class of unimodular QSIN groups (see \cite{stokke}). Since an almost connected group is inner amenable if and only if it is amenable, Thom's result equivalently states that if $G$ is a unimodular almost connected amenable group and $H$ is a discrete group which embeds into $G$, then $H$ has the factorization property. This result was generalized by Ruan and the author in \cite[Theorem 3.2]{ruanw} where it was shown that if $G$ is amenable and $H$ is discrete, then $H$ has the factorization property. The main result of this section (Corollary \ref{main 1}) states that if $G$ is QSIN and has the factorization and $H$ is a locally compact group which embeds continuously into $G$, then $H$ admits the factorization property. Since every amenable locally compact group $G$ is QSIN and admits the factorization property, this result generalizes all previously known results in the case when $H$ is discrete and extends them to the case when $H$ is locally compact.

The proof of the above mentioned result begins with the following measure theoretic lemma about continuous embeddings of locally compact groups.

\begin{lemma}\label{Lem:measure}
	Let $G$ be a locally compact group and $\{e_\alpha\}\subset \mr L^1(G)_1\cap \mr C_c(G)_+$ be a net such that $\mathrm{supp}\,e_\alpha\to \{e\}$. Further suppose that $H$ is a locally compact group, $\iota\fn H\to G$ is a continuous injective embedding of $H$ into $G$, and $K_1$ and $K_2$ are compact subsets of $H$. Then 
	$$ \int_G \left[\int_{K_2}\int_{K_1} e_\alpha(\iota(h_1^{-1})x)e_\alpha(\iota(h_2^{-1})x)\,dh_1\,dh_2\right]^{\frac{1}{2}}dx\to m_H(K_1\cap K_2).$$
\end{lemma}

\begin{proof}
	Fix $\epsilon>0$ and choose a compact sets $\widetilde{K}_1\subset K_1\backslash K_2$ and $\widetilde{K}_2\subset K_2\backslash K_1$ so that $m_H([K_1\backslash K_2]\backslash \widetilde{K_1})< \frac{\epsilon^2}{4m_H(K_2)}$ and $m_H([K_2\backslash K_1]\backslash \widetilde{K_2})< \frac{\epsilon^2}{4m_H(K_1)}$. To simplify notation, we will let $K_1'=(K_1\cap K_2)\cup \widetilde K_1$ and $K_2'=(K_1\cap K_2)\cup \widetilde K_2$. Since $\iota(K_2^{-1}\widetilde K_1)$ and $\iota(K_1^{-1}\widetilde{K_2})$ are compact (and, hence, closed) subsets of $G$ which do not contain the identity, we can find an open neighbourhood $U$ of the identity in $G$ so that $UU^{-1}\cap \iota(K_2^{-1}\widetilde K_1)=UU^{-1}\cap\iota(K_1^{-1}\widetilde{K}_2)=\emptyset$. Let $\alpha_0$ be chosen large enough so that $\mathrm{supp}\, e_\alpha\subset U$ for all $\alpha\geq \alpha_0$. Then 
	$$\chi_{\widetilde K_1}(h_1)\chi_{K_2}(h_2)e_\alpha(\iota(h_1^{-1})x)e_\alpha(\iota(h_2^{-1})x)= 0$$
	and
	$$\chi_{K_1}(h_1)\chi_{\widetilde K_2}(h_2)e_\alpha(\iota(h_1^{-1})x)e_\alpha(\iota(h_2^{-1})x)= 0$$
	for $\alpha\geq \alpha_0$, $x\in G$, and $h_1,h_2\in H$ since if $\iota(h_1^{-1})x$ and $\iota(h_2^{-1})x$ both belong to $\mathrm{supp}\,e_\alpha$, then $\iota(h_2^{-1}h_1)\in UU^{-1}$. In particular,
	\begin{eqnarray*}
		&&\int_{G}\left[\int_{K_2'}\int_{K_1'} e_\alpha(\iota(h_1^{-1})x)e_\alpha(\iota(h_2^{-1})x)\,dh_1\,dh_2\right]^{\frac{1}{2}}\,dx\\
		&=&\int_{G}\left[\int_{K_1\cap K_2}\int_{K_1\cap K_2} e_\alpha(\iota(h_1^{-1})x)e_\alpha(\iota(h_2^{-1})x)\,dh_1\,dh_2\right]^{\frac{1}{2}}\,dx\\
		&=& m_H(K_1\cap K_2)
	\end{eqnarray*}
	when $\alpha\geq \alpha_0$. So
	\begin{eqnarray*}
		&& \left|\int_G \left[\int_{K_2}\int_{K_1} e_\alpha(\iota(h_1^{-1})x)e_\alpha(\iota(h_2^{-1})x)\,dh_1\,dh_2\right]^{\frac{1}{2}}dx- m_H(K_1\cap K_2)\right|\\
		&\leq& \left|\int_G\left(\left[\int_{K_1} e_\alpha(\iota(h_1^{-1})x)\,dh_1\right]^{\frac{1}{2}}-\left[\int_{K_1'} e_\alpha(\iota(h_1^{-1})x)\,dh_1\right]^{\frac{1}{2}}\right)\left[\int_{K_2} e_\alpha(\iota(h_1^{-1})x)\,dh_2\right]^{\frac{1}{2}}dx\right|\\
		&&+\left|\int_G\left[\int_{K_1'} e_\alpha(\iota(h_1^{-1})x)\,dh_2\right]^{\frac{1}{2}}\left(\left[\int_{K_2} e_\alpha(\iota(h_1^{-1})x)\,dh_1\right]^{\frac{1}{2}}-\left[\int_{K_2'} e_\alpha(\iota(h_1^{-1})x)\,dh_1\right]^{\frac{1}{2}}\right)dx\right|.
	\end{eqnarray*}
	for $\alpha\geq \alpha_0$. Observe that
	\begin{eqnarray*}
		&&\int_G\left(\left[\int_{K_i} e_\alpha(\iota(h^{-1})x)\,dh\right]^{\frac{1}{2}}-\left[\int_{K_i'} e_\alpha(\iota(h^{-1})x)\,dh\right]^{\frac{1}{2}}\right)^2 dx\\
		&\leq& \int_G \left(\int_{K_i} e_\alpha(\iota(h^{-1})x)\,dh-\int_{K_i'} e_\alpha(\iota(h^{-1})x)\,dh\right) dx\\
		&=& m_H(K_i\backslash K_i')\\
		&< & \frac{\epsilon^2}{4m_H(K_{j})}
	\end{eqnarray*}
	for $i=1,2$, where
	$$ j=\left\{\begin{array}{cl}
	1,&\tn{if }i=2\\
	2,&\tn{if }i=1.
	\end{array}\right.$$
	Thus, H\"older's inequality implies that
	$$\left|\int_G \left[\int_{K_2}\int_{K_1} e_\alpha(\iota(h_1^{-1})x)e_\alpha(\iota(h_2^{-1})x)\,dh_1\,dh_2\right]^{\frac{1}{2}}dx- m_H(K_1\cap K_2)\right|< \epsilon$$
	for $\alpha\geq \alpha_0$.
\end{proof}


\begin{theorem}
	Let $H$ be a locally compact group. Suppose that $\iota\fn H\to G$ is a continuous embedding of $H$ into a QSIN group $G$. Then $\lambda_H\cdot\rho_H$ is weakly contained in $(\lambda_G\cdot \rho_G)\circ (\iota\times\iota)$.
\end{theorem}

\begin{proof}
	Let $K$ be a compact subset of $H\times H$ and $\epsilon>0$. Observe that simple functions of the form $f=\sum_{i=1}^m a_i\chi_{K_i}$, where $a_i\in\C$ and $K_i$ are pairwise disjoint subsets of $H$ for $i=1,\ldots,m$, compose a dense subspace of $\mr L^2(H)$. It therefore suffices to check that if $f$ is such a function, then we can find a function $\widetilde{f}\in \mr L^2(G)$ so that
	$$ \big|(\lambda_H\cdot\rho_H)_{f,f}(s,t)-(\lambda_G,\rho_G)_{\widetilde{f},\widetilde{f}}(\iota(s),\iota(t))\big|< \epsilon$$
	for all $(s,t)\in K$.
	
	Let $f=\sum_{i=1}^m a_i\chi_{K_i}$. The group algebra $\mr L^1(G)$ admits a bounded approximate identity $\{e_\alpha\}\subset \mr L^1(G)_1\cap \mr C_c(G)_+$ so that $\mathrm{supp}\,e_\alpha\to\{e\}$ and $\|\tau_1(s)e_\alpha-e_\alpha\|_1\to 0$ uniformly on compact subsets of $G$ since $G$ is QSIN (see \cite[Theorem 2.6]{stokke}). For each index $\alpha$, we define $f_\alpha\in \mr L^2(G)$ by
	$$ f_\alpha(x)=\sum_{i=1}^m a_i \left(\int_{K_i} e_\alpha(\iota(h^{-1}x))\,dh\right)^{\frac{1}{2}}. $$
	Set $M=\sum_{i=1}^m |a_i| m_H(K_i)^{\frac{1}{2}}$. Then $\|f\|_2\leq M$ and $\|f_\alpha\|\leq M$ for every index $\alpha$.
	Since $K$ is compact, we can find $(s_1,t_1),\ldots (s_n,t_n)\in K$ and $E_1,\ldots,E_n\subset K$ such that $K=\bigcup_{j=1}^n E_j$ and
	\begin{equation}\label{Eqn:a}
	\int_H\left|\chi_{K_i}(s^{-1}ht)\Delta_H(t)-\chi_{K_i}(s_j^{-1}ht_j)\Delta_H(t_j)\right|dh<\left(\frac{\epsilon}{3|a_i|mM}\right)^2
	\end{equation}
	for all $(s,t)\in E_j$, $j=1,\ldots,n$, and $i=1,\ldots,m$.
	Then
	\begin{eqnarray*}
		&&\|\lambda_H(s)\rho_H(t)f-\lambda_H(s_j)\rho_H(t_j)f\|_2\\
		&\leq&\sum_{i=1}^m |a_i| \left(\int_H\left|\chi_{K_i}(s^{-1}ht)\Delta_H(t)^{\frac{1}{2}}-\chi_{K_i}(s_j^{-1}ht_j)\Delta_H(t_j)^{\frac{1}{2}}\right|^2\right)^{\frac{1}{2}}\\
		&\leq& \sum_{i=1}^m |a_i| \left(\int_H\left|\chi_{K_i}(s^{-1}ht)\Delta_H(t)-\chi_{K_i}(s_j^{-1}ht_j)\Delta_H(t_j)\right|\right)^{\frac{1}{2}}\\
		&<& \frac{\epsilon}{3M}
	\end{eqnarray*}
	implies that
	\begin{equation}\label{Eqn:b}|(\lambda_H\cdot \rho_H)_{f,f}(s,t)-(\lambda_H\cdot \rho_H)_{f,f}(s_j,t_j)|<\frac{\epsilon}{3}\end{equation}
	for all $(s_j,t_j)\in K$ and $1\leq j\leq n$. We next prove an analogous statement for $(\lambda_G\cdot \rho_G)\circ (\iota\times\iota)$.
	
	\begin{claim}\label{Claim:a}
		There exists an index $\alpha_0$ so that
		$$ |(\lambda_G\cdot \rho_G)_{f_\alpha,f_\alpha}(\iota(s),\iota(t))-(\lambda_G\cdot \rho_G)_{f_\alpha,f_\alpha}(\iota(s_j),\iota(t_j))|<\frac{\epsilon}{3}$$
		for $(s,t)\in E_j$ and $1\leq j\leq m$ when $\alpha\geq \alpha_0$.
	\end{claim}
	
	\begin{claimproof}
		For each $y\in G$ we will let $f_{\alpha,y}\in \mr L^2(G)$ be defined by
		$$ f_{\alpha,y}(x)=\sum_{i=1}^m a_i \left(\int_{K_i} \tau_1(y)e_\alpha(\iota(h^{-1})x)\,dh\right)^{\frac{1}{2}}. $$
		Then
		\begin{eqnarray*}
			&&\|f_{\alpha,y}-f_\alpha\|_2\\
			&\leq & \sum_{i=1}^m |a_i| \left(\int_G \left|\left(\int_{K_i} \tau_1(y)e_\alpha(\iota(h^{-1}x))\,dh\right)^{\frac{1}{2}}-\left(\int_{K_i} e_\alpha(\iota(h^{-1})x)\,dh\right)^{\frac{1}{2}}\right|^2 dx\right)^{\frac{1}{2}}\\
			&\leq & \sum_{i=1}^m|a_i|\left(\int_G\int_{K_i}\left| \tau_1(y)e_\alpha(\iota(h^{-1}x))-e_\alpha(\iota(h^{-1})x)\right|\,dh\,dx\right)^{\frac{1}{2}}\\
			&=& \sum_{i=1}^m |a_i|m_H(K_i)^{\frac{1}{2}}\|\tau_1(y)e_\alpha-e_\alpha\|_1^{\frac{1}{2}}\\
			&\to & 0
		\end{eqnarray*}
		uniformly on compact subsets of $G$. Therefore, it suffices to check that
		$$ |(\lambda_G\cdot \rho_G)_{f_{\alpha,\iota(t^{-1})},f_\alpha}(\iota(s),\iota(t))-(\lambda_G\cdot \rho_G)_{f_{\alpha,\iota(t_j^{-1})},f_\alpha}(\iota(s_j),\iota(t_j))|<\frac{\epsilon}{3}$$
		for all $(s,t)\in E_j$, $1\leq j\leq n$, and every index $\alpha$.
		
		Let $(s,t)\in E_j$ for some $1\leq j\leq n$. Then		
		\begin{eqnarray*}
			&&\left\|\lambda_G(\iota(s))\rho_G(\iota(t))f_{\alpha,\iota(t^{-1})}-\lambda_G(\iota(s_j))\rho_G(\iota(t_j))f_{\alpha,\iota(t_j^{-1})}\right\|_2\\
			&\leq& \sum_{i=1}^m |a_i|\left[\int_G\left|\left(\int_{K_i}e_\alpha(\iota(sht^{-1})^{-1}x)\,dh\right)^{\frac{1}{2}}-\left(\int_{K_i}e_\alpha(\iota(s_jht_j^{-1})^{-1}x)\,dh\right)^{\frac{1}{2}}\right|^2 dx\right]^{\frac{1}{2}}\\
			&=& \sum_{i=1}^m |a_i|\left[\int_G\left|\left(\int_{H}\chi_{K_i}(s^{-1}ht)e_\alpha(\iota(h^{-1})x)\Delta_H(t)\,dh\right)^{\frac{1}{2}}\right.\right.\\
			&&\hspace{10em}\left.\left.-\left(\int_{H}\chi_{K_i}(s_j^{-1}ht_j)e_\alpha(\iota(h^{-1})x)\Delta_H(t_j),dh\right)^{\frac{1}{2}}\right|^2 dx\right]^{\frac{1}{2}}\\
			&\leq& \sum_{i=1}^m |a_i| \left[\int_G\int_{H}\left|\chi_{K_i}(sht^{-1})\Delta_H(t)-\chi_{K_i}(s_jht_j^{-1})\Delta_H(t_j)\right|e_\alpha(\iota(h^{-1})x)\,dh\,dx\right]^{\frac{1}{2}}\\
			&<& \frac{\epsilon}{3M}.
		\end{eqnarray*}
		by equation \ref{Eqn:a}. So
		$$ |(\lambda_G\cdot \rho_G)_{f_{\alpha,\iota(t^{-1})},f_\alpha}(\iota(s),\iota(t))-(\lambda_G\cdot \rho_G)_{f_{\alpha,\iota(t_j^{-1})},f_\alpha}(\iota(s_j),\iota(t_j))|<\frac{\epsilon}{3}$$
		for all $(s,t)\in E_j$, $1\leq j\leq n$, and every index $\alpha$ and, hence, we have shown the claim.
	\end{claimproof}
	
	We are now equipped to complete our proof. Combining equation \ref{Eqn:b} with claim \ref{Claim:a}, it suffices to check that there exists an index $\alpha_1$ so that
	$$ \big|(\lambda_H\cdot\rho_H)_{f,f}(s_j,t_j)-(\lambda_G,\rho_G)_{f_\alpha,f_\alpha}(\iota(s_j),\iota(t_j))\big|< \frac{\epsilon}{3}$$
	for all $1\leq j\leq n$ and $\alpha\geq \alpha_1$. Since $\|f_{\alpha,y}-f_\alpha\|_2\to 0$ for every $y\in G$, it therefore suffices to check that
	$$ (\lambda_G,\rho_G)_{f_{\alpha,\iota(t^{-1})},f_{\alpha}}(\iota(s),\iota(t))\to (\lambda_H\cdot\rho_H)_{f,f}(s,t)$$
	for each $1\leq j\leq n$ and $s,t\in H$.
	
	Let $s,t\in H$. Then
	\begin{eqnarray*}
		&&(\lambda_G,\rho_G)_{f_{\alpha,\iota(t^{-1})},f_{\alpha}}(\iota(s),\iota(t))\\
		&=&\sum_{i=1}^m\sum_{i'=1}^m a_i\overline{a_{i'}}\int_G \left(\int_{K_{i'}}\int_{K_i}e_\alpha(\iota(sh_1t^{-1})^{-1}x)e_\alpha(\iota(h_2^{-1})x)\,dh_1\,dh_2\right)^{\frac{1}{2}}dx\\
		&=&\sum_{i=1}^m\sum_{i'=1}^m a_i\overline{a_{i'}}\Delta_H(t)^{\frac{1}{2}}\int_G \left(\int_{K_{i'}}\int_{s^{-1}K_it}e_\alpha(\iota(h_1)^{-1}x)e_\alpha(\iota(h_2^{-1})x)\,dh_1\,dh_2\right)^{\frac{1}{2}}dx\\
		&\to & \sum_{i=1}^m\sum_{i'=1}^m a_i \overline{a_{i'}} \Delta_H(t)^{\frac{1}{2}}m_H(s^{-1}K_it\cap K_{i'})\\
		&=&(\lambda_H\cdot\rho_H)_{f,f}(s,t)
	\end{eqnarray*}
	by Lemma \ref{Lem:measure}. So we conclude that $\lambda_H\cdot\rho_H$ is weakly contained in $(\lambda_G\cdot \rho_G)\circ (\iota\times\iota)$.
\end{proof}

\begin{corollary}\label{main 1}
	Let $H$ be a locally compact group. If $H$ embeds continuous into a QSIN group $G$ with the factorization property, then $H$ has the factorization property.
\end{corollary}

Recall that a locally compact group is {\it maximally almost periodic} if points in $G$ are separated by finite dimensional representations. Equivalently, $G$ is maximally almost periodic if and only if $G$ embeds continuously inside a compact group. Since amenable groups are QSIN and possess the factorization property, we obtain the following further Corollary. This special case is pointed out since Section \ref{Sec:T} provides a converse to Corollary \ref{main 1} for a special class of groups.

\begin{corollary}\label{Cor:MAP}
	Maximally almost periodic groups have the factorization property.
\end{corollary}


\section{Residual properties}\label{Sec:residual}

Recall that a locally compact group $G$ is {\it residually amenable} if for every non-identity element $s\in G$, there exists a closed normal subgroup $N$ of $G$ such that $s\not\in N$ and $G/N$ is amenable. The class of residually amenable discrete groups forms one of the largest classes of examples known to have the factorization properties. The main result of this section implies, as a special case, that the same is true for all locally compact groups.

The proof that residually amenable discrete groups have the factorization is surprisingly simple (and we include it because of its elegance). Indeed, if $G$ is a residually amenable discrete group, then there exists a decreasing family of normal subgroups $\{N_\alpha\}$ of $G$ such that  $\bigcap_\alpha N_\alpha=\{e\}$ and $G/N_\alpha$ is amenable for every $\alpha$. Then the positive definite function $1_{\Delta_{G/N_\alpha}}\circ (q_\alpha\times q_\alpha)$, where $q_\alpha\fn G\to G/N_\alpha$ is the quotient map, extends to a positive linear functional of $\mr C^*(G)\otimes_{\min}\mr C^*(G)$ for every index $\alpha$ by virtue of $\mr C^*(G/N_\alpha)\otimes_{\min}\mr C^*(G/N_\alpha)$ being a quotient of $\mr C^*(G)\otimes_{\min}\mr C^*(G)$. These positive definite functions converge pointwise to $1_\Delta$ and, thus, $1_\Delta$ extends to a positive linear functional on $\mr C^*(G)\otimes_{\min}\mr C^*(G)$. So $G$ has the factorization property. The proof of this fact for locally compact groups is more difficult.

Throughout this section we will assume that $G$ is a locally compact group and $\{N_\alpha\}_{\alpha\in A}$ is a collection of closed normal subgroups of $G$ which are indexed by a directed set $A$ such that for every neighbourhood compact subset $K$ of $G$ and neighbhourhood $U$ of the identity in $G$, there exists an index $\alpha_0$ so that $N_\alpha\cap K\subset U$ for every $\alpha\geq \alpha_0$ unless otherwise stated. The main result states that if $G/N_\alpha$ has the factorization property for every $\alpha$, then $G$ has the factorization property.

We begin by showing that if $\{N_\alpha\}$ is a decreasing family of closed normal subgroups of $G$ such that $\bigcap_{\alpha}N_\alpha=\{e\}$, then $\{N_\alpha\}$ satisfies the above conditions.

\begin{prop}
	If $\{N_\alpha\}$ is a decreasing family of closed normal subgroups of a locally compact group $G$ such that $\bigcap_\alpha N_\alpha=\{e\}$, then for every neighbourhood $U$ of the identity in $G$ and compact subset $K$ of $G$, there exists an index $\alpha_0$ such that $N_\alpha\cap K\subset U$ for every $\alpha\geq \alpha_0$.
\end{prop}

\begin{proof}
	It suffices to show that there exists an index $\alpha_0$ so that $N_{\alpha_0}\cap K\subset U$ since $N_\alpha\subset N_{\alpha_0}$ for every $\alpha\geq \alpha_0$.
	If this were not possible, then there exists an element $x_\alpha\in K\backslash U$ for every $\alpha$. Since $K\backslash U$ is compact, a subnet of $\{x_\alpha\}$ converges to an element $x\in K\backslash U$. Then $x\in N_\alpha$ for every $\alpha$ since $N_\alpha$ is closed in $G$ and $x_{\beta}\in N_\alpha$ for $\beta\geq \alpha$. So $x\in \bigcap_\alpha N_\alpha$. This contradicts the assumption that $\bigcap_\alpha N_\alpha =\{e\}$.
\end{proof}

We now proceed into the proof of this section's main result.

\begin{lemma}
	For each $t\in G$, let $\varphi_t\fn G\to G$ denote the inner automorphism $\varphi_t(s)=tst^{-1}$. Let $K$ be a compact subset of $G$. There exists an index $\alpha_0$ so that $m_{N_\alpha}\circ (\varphi_t|_{N_\alpha})=m_{N_\alpha}$ for all $t\in K$ and $\alpha\geq \alpha_0$.
\end{lemma}

\begin{proof}
By replacing $K$ with $K\cup K^{-1}\cup\{e\}$, we may assume that $K$ contains the identity and is closed under the taking of inverses.
Let $U$ be any pre-compact open symmetric neighbhourhood of the identity in $G$. As $KUUK$ is a pre-compact subset of $G$, there exists an index $\alpha_0$ so that $N_\alpha\cap KUUK\subset U$ for every $\alpha\geq \alpha_0$. For the remainder of the proof we will assume that $\alpha\geq\alpha_0$.

Define $H_\alpha=N_\alpha\cap U$. Then $H_\alpha$ is closed under the taking of inverses since $U$ is symmetric and $H_\alpha$ is closed under multiplication since
$$ H_\alpha H_\alpha\subset N_\alpha\cap UU\subset N_\alpha\cap KUUK=H_\alpha.$$
So $H_\alpha$ is an open compact subgroup of $N_\alpha$. Further, $\varphi_t(H_\alpha)=H_\alpha$ for every $t\in K$ since
$$ tH_\alpha t^{-1}=t(N_\alpha\cap U)t^{-1}\subset N_\alpha\cap KUUK=H_\alpha$$
for all $t\in K$ and $K$ is self adjoint. So $m_{N_\alpha}|_{H_\alpha}=m_{N_\alpha}\circ \varphi_t|_{H_\alpha}$ for every $t\in K$ since $(m_{N_\alpha}\circ \varphi_t)|_{H_\alpha}$ is a left Haar measure for $H_\alpha$ and $m_{N_\alpha}\circ \varphi_t(H_\alpha)=m_{N_\alpha}(H_\alpha)<\infty$. Therefore, for every measurable subset $E\subset N_\alpha$ and $t\in K$,
\begin{eqnarray*}
&&m_{N_\alpha}(E)\\
&=&\sum_{\dot s\in N_\alpha/H_\alpha} m_{N_\alpha}(s^{-1}E\cap H_\alpha)\\
&=& \sum_{\dot s\in N_\alpha/H_\alpha}(m_{N_\alpha}\circ \varphi_t)(s^{-1}E\cap H_\alpha)\\
&=& \sum_{\dot s\in N_\alpha/H_\alpha}m_{N_\alpha}(\varphi_t(s^{-1})\varphi_t(E)\cap H_\alpha)\\
&=& \sum_{\dot s \in N_\alpha/H_\alpha} m_{N_\alpha}(s^{-1}\varphi_t(E)\cap H_\alpha)\\
&=& m_{N_\alpha}(\varphi_t(E))
\end{eqnarray*}
since $\varphi_t(s_1)H_\alpha=\varphi_t(s_2)H_\alpha$ if and only if $s_1 H_\alpha=s_2H_\alpha$ for $s_1,s_2\in N_\alpha$.
\end{proof}

\begin{lemma}
	Let $K$ be a compact subset of $G$. There exists an index $\alpha_0$ so that $\Delta_G(s)=\Delta_{G/N_\alpha}(\dot s)$ for all $s\in K$ and $\alpha\geq \alpha_0$.
\end{lemma}

\begin{proof}
	By taking $\alpha$ to be large enough, we may assume that $m_{N_\alpha}\circ \varphi_t=m_{N\alpha}$ for all $t\in K$. Then, for every $f\in \mr C_c(G)$ and $t\in K$,
	\begin{eqnarray*}
	&&\Delta_G(t)^{-1}\int_G f(x)\,dx\\
	&=& \int_{G}f(xt)\,dx\\
	&=& \int_{G/N_\alpha}\int_{N_\alpha} f(xht)\,dh\,d\dot x\\
	&=& \int_{G/N_\alpha}\int_{N_\alpha} f(x\varphi_{t}(h)t)\,dh\,d\dot x\\
	&=& \int_{G/N_\alpha}\int_{N_\alpha} f(xth)\,dh\,d\dot x\\
	&=& \Delta_{G/N_\alpha}(\dot t)^{-1} \int_{G/N_\alpha}\int_{N_\alpha} f(xh)\,dh\,d\dot x\\
	&=& \Delta_{G/N_\alpha}(\dot t)^{-1}\int_G f(x)\,dx.
	\end{eqnarray*}
So $\Delta_G(t)=\Delta_{G/N_\alpha}(\dot t)$ for all $t\in K$ when $\alpha$ is sufficiently large.
\end{proof}

\begin{theorem}
	 For each $\alpha$, let $\sigma_\alpha\fn G\times G\to \mr B(\mr L^2(G/N_\alpha))$ be the representation defined by
	$$\sigma_\alpha(s,t)=\lambda_{G/N_\alpha}(\dot s)\rho_{G/N_\alpha}(\dot t).$$
	Then $\sigma_\alpha$ converges to $\lambda_G\cdot\rho_G$ in the Fell topology.
\end{theorem}

\begin{proof}
	Let $f\in \mr C_c(G)$ be nonzero and $K:=\supp\,f$. By renormalizing the measures $m_{N_\alpha}$ if necessary, we may assume that $m_{N_\alpha}(N_\alpha\cap K^{-1}K)=1$ for every index $\alpha$. Define $f_\alpha\in \mr C_c(G/N_\alpha)$ by
	$$ f_\alpha(\dot x)=\int_{N_\alpha}f(xh)\,dh.$$
	We will show that $(\sigma_{\alpha})_{f_\alpha,f_\alpha}\to (\lambda_G\cdot \rho_G)_{f,f}$ in the topology of uniform convergence on compact subsets of $G$.
	
	Let $\epsilon>0$. Since $f$ is uniformly continuous, there exists an neighbourhood $U$ of the identity in $G$ such that $|f(x)-f(xy)|<\epsilon$ for all $x\in G$ and $y\in U$. By our assumption on $\{N_\alpha\}$, we can find an index $\alpha_0$ so that $N_\alpha\cap K^{-1}K\subset U$ for all $\alpha\geq \alpha_0$. Thus, for $x\in K$ and $\alpha\geq \alpha_0$, we have that
	$$ |f_\alpha(\dot x)-f(x)|=\left|\int_{N_\alpha\cap K^{-1}K}f(xh)\,dh-\int_{N_\alpha\cap K^{-1}K}f(x)\,dh\right|< \epsilon.$$
	
	Let $K'$ be a compact subset of $G$ and $q_\alpha\fn G\to G/N_\alpha$ be the canonical quotient map. We can find an index $\alpha_1\geq \alpha_0$ such that $N_\alpha \cap (K')^{-1}K^{-1}K'K\subset U$. Then, for $\alpha\geq \alpha_1$,
	\begin{eqnarray*}
		&&\left|(\sigma_\alpha)_{f_\alpha,f_\alpha}(s,t)-(\lambda_G\cdot\rho_G)_{f,f}(s,t)\right|\\
		&=& \left|\int_{G/N_\alpha}\left(\int_{N_\alpha} f(s^{-1}xth) \Delta_{G/N_{\alpha}}(t)^{\frac{1}{2}}\,dh\right)\overline{f_\alpha(\dot x)}\,d\dot x- \int_{G/N_\alpha}\int_{N_\alpha} f(s^{-1}xth)\Delta_{G}(t)^{\frac{1}{2}}\overline{f(xh)}\,dh\right|\\
		&=&	\left| \int_{q_\alpha(K)}\int_{N_\alpha\cap (K')^{-1}K^{-1}K'K}f(s^{-1}xth)\Delta_G(t)^{\frac{1}{2}}\left(\overline{f_\alpha(\dot x)}-\overline{f(xh)}\right)dh\,d\dot x\right|\\
		&\leq & 2\epsilon \|f\|_1
	\end{eqnarray*}
	since
	$$ |f_\alpha(\dot x)-f(xh)|\leq |f_\alpha(\dot x)-f(x)|+|f(x)-f(xh)|<2\epsilon$$
	for every $x\in K$ and $h\in U$. So we deduce that $\sigma_\alpha$ converges to $\lambda_G\cdot\rho_G$ in the Fell topology.
\end{proof}

\begin{corollary}\label{main 2}
	If $G/N_\alpha$ has the factorization property for every $\alpha$, then $G$ has the factorization property.
\end{corollary}

Consequently, every residually amenable group has the factorization property. In fact we can get a more general result by combining Corollary \ref{main 1} with Corollary \ref{main 2}.

We will call a locally compact group {\it residually amenably embeddable} if for every $x\in G\backslash\{e\}$, there exists a closed normal subgroup $N$ of $G$ so that $G/N$ continuously embeds inside an amenable group and $x\not\in N$. Equivalently, $G$ is residually amenably embeddable if and only if for every $x\in G\backslash\{e\}$ there exists a continuous group homomorphism $\varphi\fn G\to G'$ into an amenable group $G'$ so that $\varphi(x)\neq e$.

\begin{corollary}\label{Cor:ResAmEmb}
	If $G$ is residually amenably embeddable, then $G$ admits the factorization property.
\end{corollary}

\begin{proof}
	Suppose that $N_1$ and $N_2$ are closed normal subgroups of $G$ such that $G/N_1$ and $G/N_2$ are amenably embeddable, and fix continuous embeddings $\iota_1\fn G/N_1\to G_1$ and $\iota_2 \fn G/N_2\to G_2$ into amenable groups $G_1$ and $G_2$. Then the map
	$$ G\ni x\mapsto (\iota_1(xN_1),\iota_2(xN_2))\in G_1\times G_2$$
	is a continuous group homomorphism into $G_1\times G_2$ with kernel $N_1\cap N_2$. It follows that $G/(N_1\cap N_2)$ is amenably embeddable. Therefore
	$$A:=\{N\subset G: N\tn{ is a closed normal subgroup of $G$ so that $G/N$ is amenably embeddable}\}$$
	can be ordered into being a decreasing family of normal subgroups with the property that $\bigcap A=\{e\}$ and $G/N$ has the factorization property for every $N\in A$.
\end{proof}

The author is grateful to Nico Spronk for pointing out the following example which shows that the property of being residually amenably embeddable is strictly more general than that of being residually amenable.

\begin{example}
	Let $G$ be the group $\mr{SO}(3)$ endowed with the discrete topology. We recall that a compact group is simple if and only if it is topologically simple (see \cite[Theorem 9.90]{hoffm}). So $G$ is simple by virtue of $\mr{SO}(3)$ being topologically simple. In particular, $G$ is amenably embeddable but not residually amenable.
\end{example}

\section{Property (T) groups with the factorization property}\label{Sec:T}

One of the most celebrated results on groups with the factorization property is Kirchberg's characterization for discrete groups with property (T) (see \cite{kirch2}). This result states that a discrete group $G$ with property (T) admits the factorization if and only if $G$ is residually finite. In this section, we generalize Kirchberg's result to the context of SIN groups with property (T) by showing that such a group admits the factorization property if and only if $G$ is maximally periodic. This legitimately generalizes Kirchberg's result since a finitely generated maximally periodic group is necessarily residually finite by Mal'cev's theorem. The argument used below is a simple adaptation of Ozawa's proof of Kirchberg's aforementioned result (see \cite[Theorem 7.4]{ozawa}).

\begin{theorem}
	Suppose that $G$ is a SIN group with property (T). Then $G$ admits the factorization property if and only if $G$ is maximally almost periodic.
\end{theorem}

\begin{proof}
	Let $s\in G\backslash\{e\}$. We will show that $G$ admits a finite dimensional representation $\pi$ such that $\pi(s)\neq 1$.
	
	Since $G$ is a SIN group, we can find a compact symmetric central neighbourhood $K$ of the identity $e$ in $G$ such that $s\not \in KK$. Consider the positive definite function $u$ on $G\times G$ defined by
	$$u(s,t)=\lla\lambda(s)\rho(t)\chi_K,\chi_K\rra.$$
	Then $u$ extends to a tracial state on $\mr C^*(G)\otimes_{\min} \mr C^*(G)$ because $G$ has the factorization property. Let $\theta=\bigoplus_{\pi\in \widehat{G}}\pi$. Since $\lambda\cdot\rho$ is weakly contained in $\theta\times\overline{\theta}\fn G\times G\to \mr B(\Hi_\theta\otimes \overline{\Hi}_\theta)$, there exists unit vectors $\xi_\alpha\in (\Hi_\theta\oplus \overline{\Hi}_\theta)^{\oplus\infty}$ such that
	$$ \lla (\theta\times \overline{\theta})^{\oplus\infty}(\cdot,\cdot)\xi_\alpha,\xi_\alpha\rra\to u$$
	uniformly on compact subsets of $G\times G$. So
	$$ \|(\theta\times \overline{\theta})^{\oplus\infty}(s,s)\xi_\alpha-\xi_\alpha\|\to 0 $$
	uniformly on compact subsets of $G$ since $u(s,s)=1$ for every $s\in G$. Hence, we may assume $(\theta\otimes \overline{\theta})^{\oplus\infty}(s)\xi_\alpha=(\theta\times \overline{\theta})^{\oplus\infty}(s,s)\xi_\alpha=\xi_\alpha$ for every $\alpha$ since approximately invariant vectors are close to invariant vectors by virtue of $G$ having property (T). A well known result of Schur states that if $\pi$ and $\pi'$ are irreducible representations of a locally compact group such that $\pi\otimes\overline{\pi'}$ admits a nontrivial invariant vector, then $\pi$ and $\pi'$ are unitarily equivalent and finite dimensional. It follows that $G\times G$ admits finite dimensional representations $\{\pi_\alpha\}$ with invariant unit vectors $\eta_\alpha$ such that
	$$ \lla \pi_\alpha(\cdot,\cdot)\eta_\alpha,\eta_\alpha\rra \to u$$
	uniformly on compact subsets of $G\times G$ as $n\to \infty$. In particular,
	$$\lla\pi_\alpha(s,e)\eta_\alpha,\eta_\alpha\rra\to u(s,e)=0$$
	and, hence, $\pi_\alpha(s,e)\neq 1$ when $\alpha$ is sufficiently large.	
\end{proof}

\begin{remark}
	Zsolt Tanko has shown the author that QSIN groups with property (T) are necessarily SIN groups. As such, the term ``SIN'' in the previous theorem can be replaced with ``QSIN''.
\end{remark}

\begin{remark}
	The previous theorem does not generalize to the class of all locally compact groups. Indeed, the group $\mr {SL}(3,\R)$ has property (T) and the factorization property, but $\mr{SL}(3,\R)$ is not maximally almost periodic since it admits no nontrivial finite dimensional irreducible representations.
\end{remark}

\section{Some remaining problems}

We finish off this paper by posing two problems about the factorization property.

\begin{prob}
	Does the factorization property pass to closed subgroups?
\end{prob}

This problem is a special case of the question of whether the factorization property passes to continuous embeddings. Though Thom showed that the factorization property does not pass to continuous embeddings, his example did not arise from a closed subgroup of a group with the factorization property and, so, this special case remains open. The author suspects that the answer to this problem should also be no.

\begin{prob}
	Let $G$ be a discrete group (resp. QSIN group or inner amenable group). Is $G$ necessarily residually amenably embeddable?
\end{prob}

This question asks whether the converse of Corollary \ref{Cor:ResAmEmb} is true for particular classes of locally compact groups. We note that the converse of Corollary \ref{Cor:ResAmEmb} is false in general. Indeed, if the converse were true in general then every continuous embedding of a group with the factorization property would have the factorization property by virtue of being residually amenably embeddable. Thanks to Thom's example, we know that this is not true.

\end{document}